\documentclass[11pt]{amsart}
\usepackage[margin=1.35in]{geometry}
\usepackage{helvet, color}
\usepackage{amscd,amsmath,amsxtra,amsthm,amssymb,stmaryrd,xr,mathrsfs,mathtools,enumerate,commath, comment}
\usepackage[dvipsnames]{xcolor}
\usepackage{hyperref}
\hypersetup{
 colorlinks=true,
 linkcolor=blue,
 filecolor=magenta,
 citecolor=Aquamarine,
 urlcolor=cyan,
 pdftitle={Statistics for Galois Deformation rings},
 }
 \usepackage{tikz-cd}
\usepackage[dvipsnames]{xcolor}
\usepackage[all,cmtip]{xy}
\usepackage{stmaryrd}
\usepackage[mathscr]{eucal}
\usepackage{indentfirst}
\usepackage{graphicx}
\usepackage{tikz-cd}
\usepackage{graphics}
\usepackage{pict2e}
\usepackage{epic}
\usepackage[normalem]{ulem}
\usepackage[OT2,T1]{fontenc}
\numberwithin{equation}{section}
 
\usepackage{color}
\DeclareSymbolFont{cyrletters}{OT2}{wncyr}{m}{n}
\DeclareMathSymbol{\Sha}{\mathalpha}{cyrletters}{"58}

\newcommand{\hGL}[1]{\widehat{\operatorname{GL}_2}(#1)}

\newcommand{\Z}{\mathbb{Z}}
\newcommand{\angp}{\langle p\rangle}
\newcommand{\Dang}{\op{D}_{\bar{\rho}}^{\angp}}
\newcommand{\Rang}{\op{R}_{\bar{\rho}}^{\angp}}

\newcommand{\Q}{\mathbb{Q}}
\newcommand{\g}{\operatorname{Ad}^0\bar{\rho}}

\newcommand{\op}[1]{\operatorname{#1}}
\newcommand{\F}{\mathbb{F}}
\newcommand\mtx[4] { \left( {\begin{array}{cc}
   #1 & #2 \\
   #3 & #4 \\
  \end{array} } \right)}

\theoremstyle{plain}
 \theoremstyle{definition}
\newtheorem{Th}{Theorem}[section]
\newtheorem{Lemma}[Th]{Lemma}
\newtheorem{hypothesis}[Th]{Hypothesis}
 
\newtheorem{Corollary}[Th]{Corollary}
\newtheorem{Question}[Th]{Question}
\newtheorem{Proposition}[Th]{Proposition}
\newtheorem{Remark}[Th]{Remark}
 \theoremstyle{definition}
\newtheorem{Definition}[Th]{Definition}

\begin{document}

\title[Statistics for Galois deformation rings]{Arithmetic Statistics for Galois deformation rings}
\author[A.~Ray]{Anwesh Ray}
\address[Ray]{Chennai Mathematical Institute, H1, SIPCOT IT Park, Kelambakkam, Siruseri, Tamil Nadu 603103, India}
\email{anwesh@cmi.ac.in}
  
\author[T.~Weston]{Tom Weston}
\address[Weston]{Department of Mathematics, University of Massachusetts, Amherst, MA, USA.} 
\email{weston@math.umass.edu}
  
\begin{abstract}
Given an elliptic curve $E$ defined over the rational numbers and a prime $p$ at which $E$ has good reduction, we consider the Galois deformation ring parametrizing lifts of the residual representation on the $p$-torsion group $E[p]$. The deformations considered are subject to the flat condition at $p$. For a fixed elliptic curve without complex multiplication, it is shown that these deformation rings are unobstructed for all but finitely many primes. For a fixed prime $p$ and varying elliptic curve $E$, we relate the problem to the question of how often $p$ does not divide the modular degree. Heuristics due to M.Watkins based on those of Cohen and Lenstra indicate that this proportion should be $\prod_{i\geq 1} \left(1-\frac{1}{p^i}\right)\approx 1-\frac{1}{p}-\frac{1}{p^2}$. This heuristic is supported by computations which indicate that most elliptic curves (satisfying further conditions) have smooth deformation rings at a given prime $p\geq 5$, and this proportion comes close to $100\%$ as $p$ gets larger.
\end{abstract}

\maketitle
\section{Introduction}
\par Given an elliptic curve $E$ defined over the rationals, and a prime number $p$, denote by $E[p]$ the $p$-torsion subgroup of $E(\bar{\Q})$. B.Mazur introduced deformation functors, parametrizing lifts of the residual representation \[\bar{\rho}:\op{Gal}(\bar{\Q}/\Q)\rightarrow \op{GL}_2(\F_p)\] on $E[p]$. These functors are represented by universal deformation rings and their study gained considerable momentum in \cite{wilesflt, TW, BCDT}, where the modularity of an elliptic curve $E_{/\Q}$ is established. The approach involved showing that a certain deformation ring associated to $\bar{\rho}$ is in fact isomorphic to a localized Hecke algebra associated to a space of modular forms. A result of this flavor which establishes an isomorphism between a deformation ring and a localized Hecke algebra is known as an $\op{R}=\mathbb{T}$ theorem. Here, $\op{R}$ refers to the deformation ring associated to a residual representation $\bar{\rho}$ and $\mathbb{T}$ is the associated localized Hecke algebra. The Galois representations which arise from elliptic curves and modular forms satisfy additional local conditions, and hence, the deformations functors of interest are subject to a local constraint at $p$, which is defined using $p$-adic Hodge theory. There has been considerable interest in extending such $\op{R}=\mathbb{T}$ results to more general Galois representations arising from cuspidal automorphic representations of higher rank, see for instance \cite{CHT, taylor,lambetal}. For instance, a brilliant application of such automorphy results led to the resolution of the Sato-Tate conjecture, see \cite{LGHT}.
\par In this paper, we shall consider deformations with fixed determinant equal to the cyclotomic character. Given the importance of deformation rings in proving modularity results, it is of natural interest to explicitly characterize Galois deformation rings. Such investigations were carried out by N.Boston and B.Mazur, see \cite{boston1, boston2, bostonmazur}. Presentations for Galois deformation rings with prescribed local conditions are discussed in the work of G. B\"ockle \cite{bockle}. The deformation ring is presented by generators and relations. When there are no relations, the ring is smooth. If this is the case, the deformation ring/functor is said to be \textit{unobstructed} and it is in fact isomorphic to a power series ring over $\Z_p$. On the other hand, there may be relations, in which case its structure is more complicated. Such relations arise from \textit{local and global obstructions} to lifting. The local obstructions are characterized by local cohomology classes, while on the other hand, the global obstructions are characterized by global cohomology classes subject to local constraints (see \cite[Theorem 5.2]{bockle} for further details).
\par The deformation rings studied in this paper are equipped with a local condition at $p$, which we shall refer to as \textit{geometric deformation rings}. There is one such ring $\mathcal{R}_{E,p}$ associated to each pair $(E,p)$, where $E$ is a rational elliptic curve and $p$ an odd prime at which $E$ has good reduction. We study the structure of these rings on average, more precisely, how often they are unobstructed. Since the determinant is fixed throughout, the geometric deformation ring is smooth if and only if it is isomorphic to $\Z_p$. In this case, there is a unique characteristic-zero lift. We study the following questions.
\begin{Question}\label{main question}
\begin{enumerate}
    \item For a fixed elliptic curve $E$, how often is the deformation ring $\mathcal{R}_{E,p}$ unobstructed?
    \item For a fixed odd prime $p$ and $E$ varying over all rational elliptic curves with good reduction at $p$, how often is the deformation ring $\mathcal{R}_{E,p}$ unobstructed?
\end{enumerate}
\end{Question}
\par The first question was raised by Mazur in \cite[section 11]{infinitefern}, and studied by the second named author in \cite{westonmain}. In \textit{loc. cit.} one considers a Hecke eigencuspform $f$ of weight $k\geq 2$ and squarefree level. At each prime $p$, denote by $\bar{\rho}_{f,p}$ the residual representation at $p$. Then it is shown that for $100\%$ of the primes $p$, the full deformation ring associated to $\bar{\rho}_{f,p}$ (with no constraints on the determinant or local conditions) is unobstructed (and isomorphic to $\Z_p\llbracket X_1, X_2, X_3\rrbracket$). Furthermore, when $k>2$, the deformation ring is unobstructed for all but finitely many primes. These results were extended by J.Hatley to the case when the level is not squarefree, see \cite{hatley}. When $f$ is the modular form associated to an elliptic curve, the weight $k=2$ and it is still expected (but not proven except in a few cases; see \cite{ridgdill}) that the full deformation ring is unobstructed for all by finitely many primes. These results have been generalized to Hilbert modular forms by A.Gamzon \cite{gamzon}, and subsequently to regular algebraic conjugate selfdual
cuspidal (RACSDC) automorphic representations by D.Guiraud \cite{Guiraud}. We show that for a fixed elliptic curve $E_{/\Q}$ without complex multiplication, the geometric deformation ring $\mathcal{R}_{E,p}$ is isomorphic to $\Z_p$ for all but an explicit finite set of primes $p$, see Theorem \ref{section 4 main thm} and Corollary \ref{section 4 last}.

\par It is part (2) of Question \ref{main question} that is of primary interest from the point of view of arithmetic statistics. This question has not been studied before in any related context, and the answers provided in the last section of the paper are based on results proven in the first four sections. The answer to this question is the main goal of the paper and results are proven specifically with this application to arithmetic statistics in mind. We restrict ourselves to elliptic curves with good ordinary reduction at $p$, squarefree conductor and for which the residual representation $\bar{\rho}_{E,p}$ on $E[p]$ is irreducible. Furthermore, we insist that all primes $\ell$ dividing the conductor of $E$ are not $\pm 1\mod{p}$, see Definition \ref{section 5 defn}. For this set of elliptic curves, it is shown that the deformation ring $\mathcal{R}_{E,p}$ is unobstructed provided the degree of the modular parametrization $X_0(N)\rightarrow E$ is coprime to $p$. This condition has been studied in detail by M.Watkins in \cite{watkins}. Cohen-Lenstra heuristics indicate that the probability that $p$ divides the modular degree of an elliptic curve is given by the product 
\[1-\prod_{i\geq 1} \left(1-\frac{1}{p^i}\right)=\frac{1}{p}+\frac{1}{p^2}-\frac{1}{p^3}+\dots.\]
The heuristics indicate that the proportion of elliptic curves subject to the above conditions for which $\mathcal{R}_{E,p}$ is not isomorphic to $\Z_p$ becomes smaller as $p$ increases. These heurists are supported by computation which provide further insight into the problem.

\par The approach relies on showing that the \emph{dual Selmer group} associated with the deformation problem is trivial on average. These groups characterize obstructions to lifting Galois representations subject to local conditions. The strategy employed involves showing that this dual Selmer group vanishes provided the $p$-primary Bloch-Kato Selmer group vanishes, see Proposition \ref{BK to res selmer}. This argument is a purely Galois theoretic argument, which we expect should extend to more general Galois representations. The result of F.~Diamond, M.~Flach and L.~Guo (see \cite{DFG}) gives a criterion for the vanishing of these Bloch-Kato Selmer groups. The strategy used does not explicitly rely on $\op{R}=\mathbb{T}$ results.

\par Including the introduction, the article consists of five sections. In section \ref{section 2}, we discuss the essential objects of study in the paper, namely the deformation rings associated to elliptic curves. In section \ref{section 3}, we discuss presentations for deformation rings and establish a criterion for unobstructedness, in the setting in which there is a local condition at $p$. In section \ref{section 4}, we study the first of the two aforementioned questions. In this section, it is shown that given a non-CM elliptic curve, the deformation rings considered are unobstructed for all but a finite explicit set of primes. In section \ref{section 5}, the second question is studied, when $p\geq 5$ is a fixed prime and $E$ varies over a certain collection of elliptic curves. In this section the question of unobstructedness of the deformation ring is related to the modular degree. Cohen-Lenstra heuristics are supported by computations in this section. 
\subsection*{Acknowledgment} From September 2022 to September 2023, the first author's research was supported by the CRM-Simons fellowship. We would like to thank the referee for the helpful report.

\section{Deformation theory of Galois representations}\label{section 2}
\par Throughout this section, we fix an elliptic curve $E$ defined over the rationals and a prime number $p\geq 3$ such that the following simplifying assumptions are in place
\begin{enumerate}
    \item $E$ has good reduction at $p$,
    \item the residual representation $\bar{\rho}_{E,p}:\op{Gal}(\bar{\Q}/\Q)\rightarrow \op{GL}_2(\F_p)$ on the $p$-torsion points $E[p]$ is absolutely irreducible.
\end{enumerate}Fix an algebraic closure $\bar{\Q}$ of $\Q$. For a finite set of primes $\Sigma$, let $\Q_{\Sigma}\subset \bar{\Q}$ be the maximal extension of $\Q$ which is unramified at all primes $v\notin \Sigma$ and $\op{G}_{\Sigma}$ denote the Galois group $\op{Gal}(\Q_{\Sigma}/\Q)$.
\subsection{Galois deformations} \par Associated to the residual representation $\bar{\rho}=\bar{\rho}_{E,p}$ are various Galois deformation rings, which parametrize the Galois representations which lift $\bar{\rho}$. We recall how these rings are defined and the properties that characterize them. Throughout, $S$ will denote the set of primes $\ell$ at which $E$ has bad reduction. Note that by assumption, $S$ does not contain $p$. The global representation gives rise to a local one at each prime $\ell$. For every prime $\ell$, choose an algebraic closure $\bar{\Q}_\ell$ of $\Q_\ell$ and an embedding $\iota_\ell:\bar{\Q}\hookrightarrow \bar{\Q}_\ell$. Set $\op{G}_\ell$ to denote the absolute Galois group $\op{Gal}(\bar{\Q}_\ell/\Q_\ell)$ and note that $\iota_\ell$ induces an inclusion $\op{G}_\ell\hookrightarrow \op{Gal}(\bar{\Q}/\Q)$. With respect to this inclusion, we let $\bar{\rho}_{|\ell}$ be the restriction of $\bar{\rho}$ to $\op{G}_\ell$. We introduce various functors of \textit{Galois deformations} associated to the residual representation $\bar{\rho}$, some of which shall involve conditions on the local representations. First, we introduce the functor of deformations with no local constraints. Let $\mathcal{C}$ be the category of commutative complete local noetherian $\Z_p$-algebras with residue field isomorphic to $\F_p$. Such rings are referred to as \textit{coefficient rings} and are inverse limits of finite length local rings. A coefficient ring $A$ is equipped with the inverse limit topology and admits a presentation of the form
\[A\simeq \frac{\Z_p\llbracket X_1, \dots, X_n\rrbracket}{(f_1, \dots, f_m)},\] where $f_1, \dots, f_m$ are in the maximal ideal of $A$. Here, $\Z_p\llbracket X_1, \dots, X_n\rrbracket$ is the formal power series ring over $\Z_p$ in indeterminates $X_1, \dots, X_n$. Important examples of such coefficient rings include $\Z_p$, $\Z/p^N$, the ring of dual numbers $\F_p[\epsilon]/(\epsilon^2)$ and the Iwasawa algebra $\Z_p\llbracket T\rrbracket$. A coefficient ring $A$ comes equipped with a unique residue map $A\rightarrow \F_p$ which is a $\Z_p$-algebra homomorphism. A map between coefficient rings is a $\Z_p$-algebra homomorphism which commutes with residue maps. Denote by $\hGL{A}$ the subgroup of $\op{GL}_2(A)$ consisting of matrices which reduce to the identity in $\op{GL}_2(\F_p)$. Two continuous $A$-valued representations
\[\rho_1, \rho_2: \op{G}_{ S\cup \{p\}}\rightarrow \op{GL}_2(A)\]lifting $\bar{\rho}$ are strictly equivalent if they are conjugate by some matrix in $\hGL{A}$.
\begin{Definition}A deformation $\rho:\op{G}_{ S\cup \{p\}}\rightarrow \op{GL}_2(A)$ of $\bar{\rho}$ is a strict equivalence class of continuous lifts of $\bar{\rho}$.
\end{Definition}
Let $\chi$ denote the cyclotomic character. Throughout, all deformations $\rho$ of $\bar{\rho}$ are stipulated to have determinant equal to $\chi$. When we refer to a Galois deformation $\rho:\op{G}_{S\cup \{p\}}\rightarrow \op{GL}_2(A)$, it is to be understood that $\rho$ is an actual representation whose strict equivalence class is the deformation in question. 
\subsection{Deformation rings}
\par The functor of deformations
\[\op{D}_{\bar{\rho}}: \mathcal{C}\rightarrow \op{Sets}\]sends a ring $A$ to the set of deformations of $\bar{\rho}$ of the form $\rho:\op{G}_{ S\cup \{p\}}\rightarrow \op{GL}_2(A)$. Note that $S\cup \{p\}$ is the set of primes at which the deformations are allowed to be ramified, and this is suppressed in the notation. In fact, throughout the paper, we will only be considering deformations with minimal ramification in this sense. Since it is assumed that $\bar{\rho}$ is irreducible, it follows from \cite[Proposition 1]{mazur1} that there is a \textit{universal} Galois deformation
\[\rho^{\op{univ}}:\op{G}_{ S\cup\{p\}}\rightarrow \op{GL}_2(\op{R}_{\bar{\rho}})\] representing $\op{D}_{\bar{\rho}}$. This means that if $A$ is a coefficient ring and $\rho:\op{G}_{ S\cup \{p\}}\rightarrow \op{GL}_2(A)$ is a Galois deformation of $\bar{\rho}$, then there is a unique map of coefficient rings $\op{R}_{\bar{\rho}}\rightarrow A$ such that the following diagram commutes
\[ \begin{tikzpicture}[node distance = 2.5 cm, auto]
            \node at (0,0) (G) {$\op{G}_{ S\cup \{p\}}$};
             \node (A) at (3,0){$\op{GL}_2(A)$.};
             \node (B) at (3,2){$\op{GL}_2(\op{R}_{\bar{\rho}})$};
      \draw[->] (G) to node [swap]{$\rho$} (A);
       \draw[->] (B) to node{} (A);
      \draw[->] (G) to node {$\rho^{\op{univ}}$} (B);
\end{tikzpicture}\]
Here, $\op{R}_{\bar{\rho}}$ is referred to as the \textit{universal deformation ring} associated to $\op{D}_{\bar{\rho}}$.
\par Let $\op{I}_p$ be the inertia group at the prime $p$. A character $\psi:\op{G}_{\Q}\rightarrow \op{GL}_1(\Z_p)$ is \textit{geometric} if $\op{\psi}_{\restriction \op{I}_p}=\chi^{k-1}_{\restriction \op{I}_p}$, for $k\in \Z_{\geq 2}$. Fontaine and Mazur introduce the notion of a \textit{geometric} Galois representation.
\begin{Definition}
    Let $\rho:\op{G}_{\Q}\rightarrow \op{GL}_2(\bar{\Q}_p)$ be a continuous Galois representation, $\rho$ is said to be geometric if it satisfies the following conditions:
    \begin{enumerate}
        \item $\rho$ is irreducible,
        \item $\op{det}\rho$ is an odd geometric character, 
        \item $\rho$ is unramified away from a finite set of primes,
        \item the local representation $\rho_{\restriction p}$ is \textit{deRham} (see \cite[section 6]{brinonconrad} for the definition and basic properties of deRham representations).
    \end{enumerate}
\end{Definition}
When $\rho$ arises from a Hecke eigencuspform, it is necessarily geometric. Fontaine and Mazur conjectured that $2$-dimensional geometric Galois representations arise from Hecke eigencuspforms. This conjecture has been settled in all but one exceptional case (see \cite{taylorfm} and \cite{kisinfm}). One is primarily interested in Galois representations that are geometric in the above sense. Throughout, we shall impose an additional deformation condition at $p$, which will guarantee that the deformation rings parametrize geometric lifts. We will further elaborate on this in the next subsection.
\par Note that if $\rho$ is a deformation of $\bar{\rho}$ then $\rho_{\restriction \ell}$ is a deformation of $\bar{\rho}_{\restriction \ell}$. Therefore, a global deformation gives rise to a local deformation at each prime $\ell$. At each prime $\ell$, set $\operatorname{Def}_\ell(A)$ to be the set of $A$-deformations of $\bar{\rho}_{\restriction \ell}$. The association $A\mapsto \operatorname{Def}_\ell(A)$ defines a functor $\operatorname{Def}_\ell:\mathcal{C}\rightarrow \operatorname{Sets}$. Let $A$ be a coefficient ring with maximal ideal $\mathfrak{m}$. Let $I$ be an ideal in $A$, the quotient map $A\rightarrow A/I$ is said to be a \textit{small extension} if $\mathfrak{m}$ is principal and $I\cdot \mathfrak{m}=0$. \begin{Definition}\label{localdefconditions}
A functor of deformations $\mathcal{C}_\ell:\mathcal{C}\rightarrow \op{Sets}$ is a subfunctor of $\op{Def}_\ell$. Say that $\mathcal{C}_\ell$ is \textit{deformation condition} if the conditions $\eqref{dc1}-\eqref{dc3}$ stated below are satisfied:
      \begin{enumerate}
          \item\label{dc1} $\mathcal{C}_\ell(\F_p)=\{\bar{\rho}_{\restriction \ell}\}.$
          \item\label{dc2} For $i=1,2$, let $R_i\in \mathcal{C}$ and $\rho_i\in\mathcal{C}_\ell(R_i)$. Let $I_1$ be an ideal in $R_1$ and $I_2$ an ideal in $R_2$ such that there is an isomorphism $\alpha:R_1/I_1\xrightarrow{\sim} R_2/I_2$ satisfying \[\alpha(\rho_1 \;\text{mod}\;{I_1})=\rho_2 \;\text{mod}\;{I_2}.\] Let $R_3$ be the fibred product \[R_3=\lbrace(r_1,r_2)\mid \alpha(r_1\;\text{mod}\; I_1)=r_2\; \text{mod} \;I_2\rbrace\] and $\rho_1\times_{\alpha} \rho_2$ the induced $R_3$-representation, then $\rho_1\times_{\alpha} \rho_2\in \mathcal{C}_\ell(R_3)$.
          \item\label{dc3} Let $R\in \mathcal{C}$ with maximal ideal $\mathfrak{m}_R$. If $\rho\in \op{Def}_\ell(R)$ and $\rho\in \mathcal{C}_\ell(R/\mathfrak{m}_R^n)$ for all $n>0$ it follows that $\rho\in \mathcal{C}_\ell(R)$. In other words, the functor $\mathcal{C}_\ell$ is continuous.
      \end{enumerate}
      A deformation functor $\mathcal{C}_\ell$ is said to be \textit{liftable} if for every small extension $A\rightarrow A/I$ the induced map $\mathcal{C}_\ell(A)\rightarrow \mathcal{C}_\ell(A/I)$ is surjective.
      \end{Definition}
   
 Condition $\eqref{dc2}$ is referred to as the Mayer-Vietoris property. By a well-known result of Grothendieck \cite[section 18]{Mazurintro}, conditions $\eqref{dc1}$ to $\eqref{dc3}$ guarantee that $\mathcal{C}_\ell$ is representable.
 
 \begin{Definition}\label{adgaloisaction}Set $\op{Ad}^0\bar{\rho}$ to denote the Galois module whose underlying vector space consists of $2\times 2$ matrices with entries in $\F_p$, and with trace zero. The Galois action is as follows: for $g\in \op{G}_{\Q}$ and $v\in \op{Ad}^0\bar{\rho}$, 
     set $g\cdot v:=\bar{\rho}(g) v \bar{\rho}(g)^{-1}$.
     \end{Definition}
Note that $\g^*:=\op{Hom}(\g, \mu_p)$ can be identified with $\g(1)$, i.e., $\g$ twisted by the action of the mod-$p$ cyclotomic character $\bar{\chi}$. Let $\ell$ be a prime number, the functor of local deformations $\op{Def}_\ell$ is equipped with a tangent space. As a set, this is defined to be $\operatorname{Def}_\ell\left(\F_p[\epsilon]/(\epsilon^2)\right)$. It has the structure of a vector space over $\F_p$ and is in bijection with $H^1(\op{G}_\ell,\g)$. The bijection identifies a cohomology class $f$ with the deformation \[(\operatorname{Id}+\epsilon f)\bar{\rho}_{\restriction \ell}: \op{G}_{\ell}\rightarrow \op{GL}_2(\F_p[\epsilon]/(\epsilon^2)).\]
The tangent space $\mathcal{N}_\ell$ of a deformation condition $\mathcal{C}_\ell$ consists of the cohomology classes $f\in H^1(\op{G}_\ell, \g)$, such that $(\op{Id}+\epsilon f) \bar{\rho}_{\restriction \ell}\in \mathcal{C}_\ell\left(\F_p[\epsilon]/(\epsilon^2)\right)$. For $\ell\in S$, the deformation functor $\op{Def}_\ell$ is \textit{unobstructed} if it is liftable in the sense of Definition $\ref{localdefconditions}$. The following is a criterion for local unobstructedness.
\begin{Lemma}
The functor $\op{Def}_\ell$ is unobstructed if $H^0(\op{G}_\ell, \g^*)=0$.
\end{Lemma}
\begin{proof}
By the local Euler characteristic formula, $H^0(\op{G}_\ell, \g^*)$ is dual to $H^2(\op{G}_\ell, \g)$.
Let $A\rightarrow A/I$ be a small extension and $t$ a generator of the principal ideal $\mathfrak{m}_A$. Identify $\g$ with the kernel of the mod-$I$ reduction map 
\[\op{SL}_2(A)\rightarrow \op{SL}_2(A/I)\] by identifying $\mtx{a}{b}{c}{-a}$ with 
\[\op{Id}+t\mtx{a}{b}{c}{-a}=\mtx{1+ta}{tb}{tc}{1-ta}.\]
We show that \[\varrho:\op{G}_\ell\rightarrow \op{GL}_2(A/I)\] lifts to \[\tilde{\varrho}:\op{G}_\ell\rightarrow \op{GL}_2(A).\] Let $\tau:\op{G}_\ell\rightarrow \op{GL}_2(A)$ be a continuous lift of $\varrho$ (not necessarily a homomorphism), for which $\det \tau=\chi$. Such a lift does always exist. Note that for $g_1,g_2\in \op{G}_{\Q}$, \[\mathcal{O}(\varrho)(g_1,g_2):=\tau(g_1g_2)\tau(g_2)^{-1}\tau(g_1)^{-1}\] is an element of \[\g=\op{ker}\left\{\op{SL}_2(A)\rightarrow \op{SL}_2(A/I)\right\}.\] This defines a cohomology class 
$\mathcal{O}(\varrho)\in H^2(\op{G}_\ell, \g)$, which is trivial precisely when $\varrho$ lifts to an actual representation \[\tilde{\varrho}:\op{G}_\ell\rightarrow \op{GL}_2(A).\] Since $H^2(\op{G}_\ell, \g)=0$, a lift $\tilde{\varrho}$ must exist and the result follows.
\end{proof}
\subsection{The flat deformation condition}
\par The deformation rings we consider shall parametrize geometric Galois representations $\rho$ lifting $\bar{\rho}$ for which $\op{det} \rho=\chi$. This is ensured by choosing an appropriate local deformation condition $\mathcal{C}_p$ for $\bar{\rho}_{|p}$. A deformation $\varrho:\op{G}_p\rightarrow \op{GL}_2(A)$ of $\bar{\rho}_{|p}$ is said to be \textit{flat} if for each finite length quotient $A/I$ of $A$, the mod-$I$ representation  $\varrho_I:=\varrho\mod{I}$ arises from the Galois action on the generic fibre of a a finite flat group scheme over $\Z_p$. Such flat deformations are characterized by certain \textit{filtered Dieudonn\'e modules}.
\begin{Definition}
    A filtered Dieudonn\'e module $M$ is a $\Z_p$-module furnished with a decreasing, exhaustive, separated filtration $(M^i)$ of sub-$\Z_p$-modules. For each integer $i$, we are given a $\Z_p$-linear map $\varphi^i:M^i\rightarrow M$ such that, for $x\in M^{i+1}$, $\varphi^{i+1}(x)=p\varphi^i(x)$. Denote by $\op{MF}_{\op{tor}}^f$ the category of filtered Dieudonn\'e-modules $M$ that are of finite length and for which $\sum_i \varphi^i(M)=M$. For $j>0$, set $\op{MF}_{\op{tor}}^{f,j}$ to be the full subcategory of $\op{MF}_{\op{tor}}^f$ consisting of modules $M$ for which $M^0=M$ and $M^j=0$.
\end{Definition}

Let $\op{Rep}_{\Z_p}^f$ be the category of finite length $\Z_p[\op{G}_p]$-modules. For $j<p$, Fontaine and Laffaille show in \cite{fontainelaffaille} that there is a faithful, exact contravariant functor 
\[\op{U}_S: \op{MF}_{\op{tor}}^{f,j}\rightarrow \op{Rep}_{\Z_p}^f.\]
Since we fix the determinant of our deformations to be equal to $\chi$, we need only focus our attention to the category $\op{MF}_{\op{tor}}^{f,2}$. Fontaine and Laffaille show that $\op{U}_S$ induces an anti-equivalence between this category and the category of $\op{G}_p$-modules arising from finite flat group schemes over $\Z_p$. Since the elliptic curve $E$ is assumed to have good reduction at $p$, the residual representation $\bar{\rho}_{|p}$ is flat, i.e., arises from a finite flat group scheme over $\Z_p$. The result of Fontaine and Laffaille implies that there exists a uniquely determined module $M_0\in \op{MF}_{\op{tor}}^{f,2}$, such that $\op{U}_S(M_0)=V_{\bar{\rho}}$, where $V_{\bar{\rho}}$ is the underlying $\F_p$-vector space on which $\op{G}_p$ acts via $\bar{\rho}_{|p}$. Let $R\in \mathcal{C}$, and $R/I$ be a finite length quotient of $R$. The mod-$I$ reduction of $\varrho:\op{G}_p\rightarrow \op{GL}_2(R)$ is denoted $\varrho_I$ and its underlying module is denoted $V_{\varrho_I}$. We take note of this result.
\begin{Th}[Fontaine-Laffaille]
Let $R\in \mathcal{C}$, a deformation $\varrho:\op{G}_p\rightarrow \op{GL}_2(R)$ is flat if for each finite length quotient $R/I$, there is a Dieudonn\'e module $M_{\varrho_I}\in \op{MF}_{\op{tor}}^{f,2}$, such that $\op{U}_S(M_{\varrho_I})=V_{\varrho_I}$.
\end{Th}
\begin{proof}
Please refer to \cite[section 9]{fontainelaffaille}, where the result is proved.
\end{proof}
\par Let $\op{F}_2$ be the subfunctor of $\op{Def}_p$ consisting of flat deformations. Ramakrishna in \cite{ramakrishna compositio} showed that under certain hypotheses, $\op{F}_2$ is representable by a power series ring over $\Z_p$. Let us recall these results. Set $\op{End}_p(\bar{\rho})$ to denote the ring of $\op{G}_p$-module endomorphisms of $V_{\bar{\rho}}$. 
\begin{Th}[Ramakrishna]
Assume that $\op{End}_p(\bar{\rho})=\F_p$, then, $\op{F}_2$ is pro-representable  by a smooth ring $R_2\simeq \Z_p\llbracket X\rrbracket$.
\end{Th}
\begin{proof}
The result of Ramakrishna is proved in the setting where the determinants of the deformations $\varrho$ of $\bar{\rho}_{|p}$ are not fixed. Since $\varrho$ is flat, $\op{det}\varrho_{|\op{I}_p}=\chi$. Note however, that we fix the entire determinant $\op{det} \varrho=\chi$. Ramakrishna shows that the flat deformation functor without fixed determinant is representable by $\Z_p\llbracket X_1, X_2\rrbracket$. This is proven by showing that the full-adjoint tangent space corresponding to flat deformations contained in $H^1(\op{G}_p, \op{Ad}\bar{\rho})$ is $2$-dimensional, and hence, the ring is of the form $\Z_p\llbracket X_1, X_2\rrbracket/\mathcal{I}$. It is also shown that this functor is representable by a smooth ring over $\Z_p$, hence, the ideal $\mathcal{I}=0$. The same argument in the case when the determinant is fixed, shows that when $\det=\chi$ is fixed, the functor is represented by $\Z_p\llbracket X\rrbracket /\mathcal{J}$, and since the functor is liftable, it must be isomorphic to $\Z_p\llbracket X\rrbracket$.
\end{proof}
When the condition $\op{End}_p(\bar{\rho})=\F_p$ is satisfied, set $\mathcal{C}_p$ to denote the flat deformation functor with $\op{det}=\chi$. Let $\mathcal{N}_p\subseteq H^1(\op{G}_p, \g)$ denote its tangent space, note that by the above result, $\mathcal{N}_p$ is $1$-dimensional. Also note that in this setting when $\op{End}_p(\bar{\rho})=\F_p$, it follows that $H^0(\op{G}_p, \g)=0$. Therefore, the formula
\[\dim \mathcal{N}_p=1+\dim H^0(\op{G}_p, \g)\]is automatically satisfied.
When $E$ has supersingular reduction at $p$, the condition $\op{End}_p(\bar{\rho})=\F_p$ follows from \cite{serre}. On the other hand, when $E$ is ordinary at $p$, there is an unramified character $\psi$ such that $\bar{\rho}_{|p}=\mtx{\psi \bar{\chi}}{\ast}{0}{\psi^{-1}}$. In this setting, $\op{End}_p(\bar{\rho})=\F_p$ if and only if the representation $\bar{\rho}_{|p}$ is non-split. Therefore, consider the case when $\bar{\rho}_{|p}$ is split. In this case, Ramakrishna shows in \cite[Table 2, p. 128]{ravi2} that there is a pair $(\mathcal{C}_p, \mathcal{N}_p)$, where $\mathcal{C}_p$ is a liftable deformation functor with determinant equal to $\chi$ and $\mathcal{N}_p$ is the tangent space of $\mathcal{C}_p$. Moreover, these deformations are all ordinary, and are defined as follows.
\begin{itemize}
    \item When $\bar{\rho}_{|p}$ is twist equivalent to $\mtx{\bar{\chi}}{0}{0}{1}$, it is shown that there is subfunctor $\mathcal{C}_p$ of $\op{Def}_p$ consisting of flat deformations represented by $\Z_p\llbracket X_1, X_2\rrbracket$.
    \item When $\bar{\rho}_{|p}$ is \textit{not} twist equivalent to $\mtx{\bar{\chi}}{0}{0}{1}$, it is shown that there is subfunctor $\mathcal{C}_p$ of $\op{Def}_p$ consisting of ordinary deformations represented by $\Z_p\llbracket X_1, X_2\rrbracket$. In this setting, we check that all infinitesimal lifts $\mathcal{C}_p\left(\F_p[\epsilon]/(\epsilon^2)\right)$ are flat.
\end{itemize}
\begin{Proposition}\label{Np flat}
Let $E$ be an elliptic curve with good reduction at $p\geq 5$, and $\bar{\rho}$ the residual representation. Then, $\bar{\rho}_{|p}$ comes equipped with a liftable deformation condition $\mathcal{C}_p$ along with tangent space $\mathcal{N}_p$, such that
\begin{enumerate}
\item all deformations $\varrho$ satisfying $\mathcal{C}_p$ are ordinary (resp. crystalline) if $E$ has ordinary (resp. supersingular) reduction at $p$,
    \item $\dim \mathcal{N}_p=1+\dim H^0(\op{G}_p, \g)$,
    \item the deformations $\varrho\in \mathcal{C}_p\left(\F_p[\epsilon]/(\epsilon^2)\right)$ are flat, i.e., arise via $\op{U}_S$.
\end{enumerate}
\end{Proposition}
\begin{proof}
The above discussion shows that the only case that needs to be considered is when $\bar{\rho}_{|p}=\mtx{\psi \bar{\chi}}{0}{0}{\psi^{-1}}$ for an unramified character $\psi$. Let $M_0$ be the Fontaine-Laffaille module associated to $V_{\bar{\rho}}$. Then $M_0$ is a direct sum $M_0=N_1\oplus N_2$, where $N_1$ and $N_2$ are $1$-dimensional Fontaine-Laffaille modules corresponding to the Galois modules $\F_p(\psi \bar{\chi})$ and $\F_p(\psi^{-1})$ respectively. Choose a basis $\{e,k\}$ of $M_0$ such that $k$ spans $F^1 M_0$. The Fontaine-Laffaille module $M_0$ is characterized by a dashed matrix 
\[\left( {\begin{array}{c|c}
   a & c \\
   b & d \\
  \end{array} } \right)\in \text{M}_2(\F),\]
  where $\varphi^0(e)=ae+bk$ and $\varphi^1(k)=ce+dk$. The relation $\varphi^0_{\restriction M_0^1}=p\varphi^1=0$ implies that $\varphi^0(k)=0$. Therefore, the matrices corresponding to $\varphi^0$ and $\varphi^1$ are as follows \[\varphi^0=\mtx{a}{0}{b}{0}\text{, and }\varphi^1=\left( {\begin{array}{c}
   c \\
   d \\
  \end{array} } \right).\]
  
  On the other hand, according to \textit{loc. cit.}, $\mathcal{C}_p$ is the ordinary deformation condition and $\mathcal{N}_p$ is the ordinary tangent space and has dimension $2$. Explicitly, $\mathcal{N}_p\subset H^1(\op{G}_p, \g)$ has a basis of two cohomology classes $f_1$ and $f_2$, where $f_1$ is ramified and has image in $\mtx{0}{\ast}{0}{0}$, and on the other hand, $f_2$ is unramified and has image in the diagonal in $\g$. Both cohomology classes are seen to arise from Fontaine Laffaille modules $\widetilde{M}_0$ which fit into a short exact sequence 
\[0\rightarrow M_0\rightarrow \widetilde{M}_0\rightarrow M_0\rightarrow 0.\] The definition of $\widetilde{M}_0$ involves a choice of a dashed matrix with entries in $\F_p[\epsilon]/\epsilon^2$ lifting that corresponding $M_0$. We leave the details of this simple matrix calculation to the reader.
\end{proof}
\section{Presentations for Geometric Deformation Rings}\label{section 3}
Assume throughout that $p\geq 5$ and that \[\bar{\rho}:\op{G}_{S\cup \{p\}}\rightarrow \op{GL}_2(\F_p)\] is as in the previous section. In this section, we consider Galois representations which are ordinary/crystalline when localized at $p$, and thus in particular, are deRham. Such representations will be geometric and under a mild additional hypothesis, are known to arise from Hecke eigencuspforms. Following B\"ockle \cite{bockle}, we discuss presentations for the associated Galois deformation rings. We do not impose any local conditions at the primes $\ell\in S$. However, at $p$, we impose the condition $\mathcal{C}_p$ defined in the previous section. Recall that throughout, the determinant of all deformations considered is fixed and equal to $\chi$, the cyclotomic character.

\par We introduce an additional hypothesis which will play a role in simplifying some results in arithmetic statistics, however, the results in this section shall not assume this hypothesis.
\begin{hypothesis}We say that condition $(\star)$ is satisfied for $\bar{\rho}$ if for all $\ell\in S$,
\[H^0(\op{G}_\ell, \g^*)=0.\]
\end{hypothesis}

We introduce Selmer and dual Selmer groups which will play an important role in describing presentations for Galois deformation rings. First, we introduce the notion of Selmer data for the set of primes $S\cup \{p\}$. A Selmer datum $\mathcal{L}$ consists of a choice of subspace $\mathcal{L}_\ell\subseteq H^1(\op{G}_\ell, \g)$ for each prime $\ell\in S\cup \{p\}$; we set $\mathcal{L}_\infty=0$. Let $\mathcal{L}_\ell^{\perp}\subset H^1(\op{G}_p, \g^*)$ be the orthogonal complement of $\mathcal{L}_\ell$ with respect to the local Tate-duality pairing. Then, the Selmer and dual Selmer groups associated to the datum $\mathcal{L}$ are defined as follows
\begin{equation}\begin{split}
    & H^1_{\mathcal{L}}(\op{G}_{S\cup \{p\}}, \g)
    := \text{ker}\left\{ H^1(\op{G}_{S\cup \{p\}}, \g)\longrightarrow \bigoplus_{\ell\in S\cup \{p\}}\frac{H^1(\op{G}_\ell, \g)}{\mathcal{L}_\ell}\right\}\\
    & H^1_{\mathcal{L}^\perp}(\op{G}_{S\cup \{p\}}, \g^*)
    := \text{ker}\left\{ H^1(\operatorname{G}_{S\cup \{p\}}, \g^{*})\longrightarrow \bigoplus_{\ell\in S\cup \{p\}}\frac{H^1(\op{G}_\ell, \g^*)}{\mathcal{L}_\ell^{\perp}}\right\}\\
\end{split}\end{equation}
      respectively.

In this paper, we work in a special case, where $\mathcal{L}_\ell$ is the full space $H^1(\op{G}_\ell, \g)$ for $\ell\in S$ and $\mathcal{L}_p=\mathcal{N}_p$. The following notation we use is not standard and there are no conditions at the primes $\ell\in S$.
The Selmer and dual Selmer groups are as follows
\begin{equation}\begin{split}
    & H^1_{\angp}(\op{G}_{S\cup \{p\}}, \g)
    := \text{ker}\left\{ H^1(\op{G}_{S\cup \{p\}}, \g)\longrightarrow \frac{H^1(\op{G}_p, \g)}{\mathcal{N}_p}\right\}\\
    & H^1_{\angp}(\op{G}_{S\cup \{p\}}, \g^*)
    := \text{ker}\left\{ H^1(\operatorname{G}_{S\cup \{p\}}, \g^{*})\longrightarrow \frac{H^1(\op{G}_p, \g^*)}{\mathcal{N}_p^{\perp}}\right\}\\
\end{split}\end{equation}
      respectively.
      \begin{Proposition}\label{prop:selmer vs dselmer}
      With respect to notation introduced above, we have that
      \[\dim H^1_{\angp}(\op{G}_{S\cup \{p\}}, \g)-\dim H^1_{\angp}(\op{G}_{S\cup \{p\}}, \g^*)=\sum_{\ell\in S} H^0(\op{G}_\ell, \g^*).\] In particular, $(\star)$ is satisfied if and only if
      \[\dim H^1_{\angp}(\op{G}_{S\cup \{p\}}, \g)=\dim H^1_{\angp}(\op{G}_{S\cup \{p\}}, \g^*).\]
      \end{Proposition}
      \begin{proof}
      Suppose for each prime $\ell\in S\cup\{p, \infty\}$, there is a choice of subspace $\mathcal{L}_\ell\subseteq H^1(\op{G}_\ell, \g)$. The Selmer condition $\mathcal{L}$ is the data $\{\mathcal{L}_\ell\}$ and the dimensions of the Selmer and dual Selmer groups are related as follows
      \[\begin{split}&\dim H^1_{\mathcal{L}}(\op{G}_{S\cup \{p\}}, \g)-\dim H^1_{\mathcal{L}^{\perp}}(\op{G}_{S\cup \{p\}}, \g^*)\\
      =& \dim H^0(\op{G}_{\Q}, \g)-\dim H^0(\op{G}_{\Q}, \g^*)\\
      +&\sum_{\ell\in S\cup\{p, \infty\}} \left(\dim \mathcal{L}_\ell-\dim H^0(\op{G}_\ell, \g)\right),\\
      \end{split}\]
      according to Wiles' formula \cite[Theorem 8.7.9]{NW}.
      The representation $\bar{\rho}$ associated to an elliptic curve is \textit{odd}, i.e., $\det \bar{\rho}(c)=-1$, where $c$ denotes complex conjugation. It is therefore easy to show that $\dim H^0(\op{G}_\infty, \g)=1$. Now, specify the Selmer data as follows 
      \[\mathcal{L}_\ell=\begin{cases}
      0\text{ if }\ell=\infty,\\
      \mathcal{N}_p\text{ if } \ell=p,\\
      H^1(\op{G}_\ell, \g)\text{ if }\ell\in S.\\
      \end{cases}\]
      
      In this case, the Selmer group $H^1_{\mathcal{L}}(\op{G}_{S\cup \{p\}}, \g)$ coincides with the Selmer group $H^1_{\angp}(\op{G}_{S\cup \{p\}}, \g)$. Since $\g$ is assumed to be irreducible, we have that \[\dim H^0(\op{G}_{\Q}, \g)=\dim H^0(\op{G}_{\Q}, \g^*)=0.\]
      Also, note that for $\ell \neq p$ by the local Euler characteristic formula, 
      \[\dim H^1(\op{G}_\ell, \g)-\dim H^0(\op{G}_\ell, \g)=\dim H^2(\op{G}_\ell, \g).\]
      On the other hand, according to local duality, 
      \[\dim H^2(\op{G}_\ell, \g)= \dim H^0(\op{G}_\ell, \g^*) .\] Finally, note that
      \[\dim \mathcal{N}_p=1+\dim H^0(\op{G}_p, \g).\] Putting it all together, we obtain the result.
      
      \end{proof}
     For the next part of the discussion, a good reference is \cite{bockle}. For $R\in \mathcal{C}$, let $\Dang(R)$ be the subset of $\op{D}_{\bar{\rho}}(R)$ consisting of deformations $\rho$ such that the local representation $\rho_{|p}$ satisfies $\mathcal{C}_p$. The functor $\Dang$ is represented by a universal deformation 
\[\rho^{\op{univ}, \angp}:\op{G}_S\rightarrow \op{GL}_2(\Rang),\] and we shall refer to $\Rang$ as the \text{universal geometric deformation ring}. This is abuse of terminology, since it does not capture all geometric deformations, since there may indeed be characteristic zero deformations $\varrho$ of $\bar{\rho}_{|p}$ that are deRham yet do not satisfy $\mathcal{C}_p$. However, since our constructions require choosing a suitable local condition that is smooth and satisfies other suitable properties, we work in this setting. 

\par Let $\ell\in S$, and $\rho_\ell:\op{G}_\ell\rightarrow \op{GL}_2(R_\ell)$ be the universal deformation of $\bar{\rho}_{|\ell}$ representing the functor $\op{Def}_\ell$. This local deformation ring has a presentation 
\[R_\ell\simeq \frac{\Z_p\llbracket X_1, \dots, X_u\rrbracket}{\left(g_1, \dots, g_v\right)}\]
where \[u=\dim H^1(\op{G}_\ell, \g)\textit{ and }v=\dim H^2(\op{G}_\ell, \g).\] Denote the ideal of relations by $\mathcal{I}_\ell:=\left(g_1, \dots, g_v\right)$. Let $\rho_p:\op{G}_p\rightarrow \op{GL}_2(R_p)$ be the universal deformation of $\bar{\rho}_{|p}$ representing the functor $\mathcal{C}_p$. Note that since the deformation functor $\mathcal{C}_p^{\psi}$ is liftable, the local ring $R_p$ is a smooth power series ring 
\[R_p\simeq \Z_p\llbracket X_1, \dots X_u\rrbracket\]
with $u=\dim \mathcal{N}_p$ generators.
\par By universality, the local representation $\rho_{|\ell}^{\op{univ}, \angp}$ arises by composing \[\rho_\ell:\op{G}_\ell\rightarrow \op{GL}_2(R_\ell)\] by the map $\op{GL}_2(R_\ell)\rightarrow \op{GL}_2(\Rang)$ induced by a uniquely determined map of local rings $R_\ell\rightarrow \Rang$. Let $\mathcal{J}_\ell$ denote the ideal generated by the image of $\mathcal{I}_\ell$ under the image of this map. The following result is a consequence of \cite[Theorem 5.2]{bockle}.
\begin{Th}\label{theorem presentations}
There is a presentation 
\[\Rang\simeq \frac{\Z_p\llbracket X_1, \dots, X_t\rrbracket}{\mathcal{J}},\] with 
\[t:=\dim H^1_{\angp}(\op{G}_{S\cup \{p\}}, \g)\] and where $\mathcal{J}$ is generated by the ideals $\mathcal{J}_\ell$, for $\ell\in S$, and $s:=\dim H^1_{\angp}(\op{G}_{S\cup \{p\}}, \g^*)$ globally defined relations $f_1,\dots, f_s\in (p, X_1, \dots, X_t)^2$. Each ideal $\mathcal{J}_\ell$ is generated by at most $ \dim H^0(\op{G}_\ell, \g^*)$ local relations.
\end{Th}
\begin{Corollary}\label{corollary def zp}
Suppose that $(\star)$ is satisfied and 
\[H^1_{\angp}(\op{G}_{S\cup \{p\}}, \g^*)=0.\] Then, $\Rang$ is isomorphic $\Z_p$.
\end{Corollary}
\begin{proof}
It follows from Proposition \ref{prop:selmer vs dselmer} that both the Selmer group $H^1_{\angp}(\op{G}_{S\cup \{p\}}, \g)$ and the dual Selmer group $H^1_{\angp}(\op{G}_{S\cup \{p\}}, \g^*)$ are zero, hence, $s=t=0$. Furthermore, since $H^2(\op{G}_\ell, \g)=0$ for $\ell\in S$, it follows that $\mathcal{J}_\ell=0$ are all $\ell\in S$. Since $s=0$, it follows that there no globally defined relations, and hence, $\mathcal{J}=0$. Therefore, it follows from Theorem \ref{theorem presentations} that $\Rang$ is isomorphic $\Z_p$.
\end{proof}
\section{Statistics for a fixed Elliptic Curve $E_{/\Q}$ and varying prime $p$}\label{section 4}
\par Let $E$ be an elliptic curve defined over $\Q$ with squarefree conductor $N$ and $p\geq 5$ a prime. Then, the Galois action on the $p$-torsion points gives rise to the $2$-dimensional mod-$p$ Galois representation $\bar{\rho}=\bar{\rho}_{E,p}:\op{G}_{\Q, S\cup \{p\}}\rightarrow \op{GL}_2(\F_p)$. Associated to each pair $(E,p)$ such that:
\begin{enumerate}
    \item $E$ has good reduction at $p$,
    \item the Galois representation $\bar{\rho}$ is absolutely irreducible,
\end{enumerate}
let $\mathcal{R}_{E,p}=\Rang$ be the geometric deformation ring introduced in the previous section. Note that for a fixed elliptic curve $E$ without complex multiplication, all but finitely many primes $p$ satisfy the above conditions. In this section, we show that $\mathcal{R}_{E,p}\simeq \Z_p$ for all but finitely many primes $p$ (that satisfy the conditions above). This result may be contrasted with the main result of \cite{westonmain}, where it is shown that $\op{R}_{\bar{\rho}}$ is unobstructed for all but a density zero set of primes.
\par There are 2 steps to the arguments in this section. 
\begin{enumerate}
    \item It is shown that condition $(\star)$ is satisfied for all but finitely many primes $p$.
    \item It is shown that for all but finitely many $p$, $ H^1_{\angp}(\op{G}_{S\cup  \{p\}}, \g)=0$ for the residual representation $\bar{\rho}=\bar{\rho}_{E,p}$.
\end{enumerate}
\subsection{Bloch-Kato Selmer groups}
\par In order to prove our results about the vanishing of the Selmer group $H^1_{\angp}(\op{G}_{S\cup \{p\}}, \g^*)$, we relate it to the Bloch-Kato Selmer group, which is associated with the characteristic zero representation. Let us recall some standard definitions. In the next discussion, assume that the pair $(E, p)$ is fixed, and recall that $S$ is the set of primes $\ell\neq p$ at which $E$ has bad reduction. Let $\op{T}_p(E)$ denote the $p$-adic Tate-module associated to $E$ and $\op{V}_p(E)$ be the $p$-adic vector space $\op{T}_p(E)\otimes_{\Z_p} \Q_p$, equipped with the action of $\op{G}_{S\cup \{p\}}$, and let $\op{Ad}_p^0(E)$ be the adjoint representation associated to $\op{V}_p(E)$. For a $\op{G}_{S\cup \{p\}}$-module $M$, set $H^i(\Q_{S\cup\{p\}}/\Q, M)$ to denote the group of continuous classes $H^i_{\op{cnts}}(\op{G}_{S\cup\{p\}}, M)$. For each place $\ell$ of $\Q$, Bloch and Kato define a subspace \[H^1_{\bf{f}}\left(\Q_\ell, \op{Ad}_p^0(E)\right)\subseteq H^1\left(\Q_\ell, \op{Ad}_p^0(E)\right)\] as follows
\[H^1_{\bf{f}}\left(\Q_\ell, \op{Ad}_p^0(E)\right):=\begin{cases} H^1_{nr}\left(\Q_\ell, \op{Ad}_p^0(E)\right) & \text{ if }\ell\neq p, \infty,\\
\op{ker}\left\{H^1(\Q_\ell, \op{Ad}_p^0(E))\longrightarrow H^1(\Q_\ell, \op{B}_{\op{crys}}\otimes \op{Ad}_p^0(E))\right\} & \text{ if }\ell=p,\\
0 & \text{ if }\ell=\infty.
\end{cases}
\]
Let $\op{W}_p(E)$ be the quotient $\op{Ad}_p^0(E)\otimes_{\Q_p} \Q_p/\Z_p$; the natural quotient map $\op{Ad}_p^0(E)\rightarrow \op{W}_p(E)$ induces a map on passing to cohomology
\[H^1\left(\Q_\ell, \op{Ad}_p^0(E)\right)\rightarrow H^1\left(\Q_\ell, \op{W}_p(E)\right).\]
Let $\ell$ be any prime, set
\[H^1_{\bf{f}}\left(\Q_\ell, \op{W}_p(E)\right):=\op{im}\left\{H^1_{\bf{f}}\left(\Q_\ell, \op{Ad}_p^0(E)\right)\longrightarrow H^1(\Q_\ell, \op{W}_p(E))\right\}.\] For $M$ denoting either $\op{Ad}_p^0(E)$ or $\op{W}_p(E)$, the Bloch-Kato Selmer group associated to $M$ is defined as follows
\[H^1_{\bf{f}}(\Q_{S\cup \{p\}}/\Q, M):=\ker\left\{H^1(\Q_{S\cup \{p\}}/\Q, M)\longrightarrow \bigoplus_\ell \frac{H^1(\Q_\ell, M)}{H^1_{\bf{f}}(\Q_\ell, M)}\right\}.\]
For $\ell\notin \{p, \infty\}$, the local cohomology group $H^1_{\bf{f}}\left(\Q, \op{W}_p(E)\right)$ coincides with the maximal divisible subgroup of $H^1_{\op{nr}}\left(\Q_\ell, \op{W}_p(E)\right)$. For any finite set of primes $\Sigma$ not containing $p,\infty$, define a larger Selmer group as the kernel of the following map
\[H^1_\Sigma(\Q, M)\longrightarrow \left(\bigoplus_{\ell\notin \Sigma\cup \{p,\infty\}}\frac{H^1(\Q_\ell, M)}{H^1_{\op{nr}}(\Q_\ell, M)}\right)\oplus \left(\frac{H^1(\Q_p, M)}{H^1_{\bf{f}}(\Q_p, M)}\right).\]
Note that the inclusion $H^1_{\bf{f}}\left(\Q, \op{W}_p(E)\right)\subseteq H^1_{\emptyset}\left(\Q, \op{W}_p(E)\right)$ is an equality if $H^0\left(\op{I}_\ell,\op{W}_p(E)\right)$ is divisible for all $\ell\neq p$.
\subsection{Vanishing of the dual Selmer group}
We now study the relationship between the Bloch-Kato Selmer group and the smoothness of the geometric deformation ring. In order to arrive at such a relationship, we establish a criterion for the vanishing of the dual Selmer group $H^1_{\angp}(\op{G}_{S\cup \{p\}}, \g^*)$.
\begin{Lemma}\label{lemma 4.1}
Let $\mathcal{N}_p\subseteq H^1(\op{G}_p, \g)$ be the tangent space of $\mathcal{C}_p$. Then, the image of $\mathcal{N}_p$ under the natural map 
\[H^1(\op{G}_p, \g)\rightarrow H^1\left(\op{G}_p, \op{W}_p(E)\right)[p]\] is contained in $H^1_{\bf{f}}\left(\op{G}_p, \op{W}_p(E)\right)[p]$.
\end{Lemma}
\begin{proof}
It is well known that elements of $ H^1_{\bf{f}}\left(\op{G}_p, \op{W}_p(E)\right)$ are precisely those classes whose corresponding extensions lie in the image of the Fontaine-Laffaille functor, see the discussion in \cite[section 1.1.2 and p.697]{DFG}. On the other hand, it follows from part (3) of Proposition \ref{Np flat} that all cohomology classes $f\in \mathcal{N}_p$ have this property. Hence, any class $f\in \mathcal{N}_p$ must map to $H^1_{\bf{f}}\left(\op{G}_p, \op{W}_p(E)\right)$.
\end{proof}
\begin{Proposition}\label{BK to res selmer}
Let $E$ be an elliptic curve defined over $\Q$ and $p\geq 5$ a prime at which $E$ has good reduction. Assume that the following conditions hold:
\begin{enumerate}
    \item the mod-$p$ Galois representation $\bar{\rho}_{E,p}$ is absolutely irreducible,
    \item $H^1_\emptyset(\Q, \op{W}_p(E))=0$.
\end{enumerate}Then the following assertions hold:
\begin{enumerate}
    \item $\dim H^1_{\angp}(\op{G}_{ S\cup \{p\}}, \g)=\sum_{\ell\in S} \dim H^0(\op{G}_\ell, \g^*)$,
    \item $H^1_{\angp}(\op{G}_{S\cup \{p\}}, \g^*)=0$.
\end{enumerate}
\end{Proposition}
\begin{proof}
Throughout this proof, we set $\bar{\rho}:=\bar{\rho}_{E,p}$. Identify $\op{W}_p(E)[p]$ with $\g$, and note that since $\bar{\rho}$ is assumed to be irreducible, it follows that $H^0(\Q, \g)=0$, and as a result, the natural map 
\begin{equation}\label{natural map}H^1(\Q, \g)\longrightarrow H^1(\Q, \op{W}_p(E))\end{equation} is injective. Let $\mathcal{T}\subset H^1_{\angp}(\op{G}_{S\cup\{p\}}, \g)$ consist of cohomology classes $f$ that are unramified at all primes $\ell\neq p$. In other words,
\[\mathcal{T}:=\op{ker}\left\{ H^1(\op{G}_{S\cup\{p\}}, \g)\longrightarrow \left(\bigoplus_{\ell\in S} \frac{H^1(\op{G}_\ell, \g)}{H^1(\op{G}_\ell/\op{I}_\ell, (\g)^{\op{I}_\ell})}\right)\oplus \left( \frac{H^1(\op{G}_p, \g)}{\mathcal{N}_p}\right) \right\}.\]We show that the above injection \eqref{natural map} restricts to a map
\[\mathcal{T}\longrightarrow H^1_\emptyset\left(\Q, \op{W}_p(E)\right).\] It suffices to observe that the local conditions defining $\mathcal{T}$ as a subspace of $H^1(\Q, \g)$ are compatible with those defining $H^1_\emptyset\left(\Q, \op{W}_p(E)\right)$. Clearly, this is the case for $\ell\neq p$, and for $\ell=p$, the assertion follows from Lemma \ref{lemma 4.1}. As a result, the assumption that $H^1_\emptyset\left(\Q, \op{W}_p(E)\right)=0$ implies that $\mathcal{T}=0$.
On the other hand, consider the short exact sequence
\[0\rightarrow \mathcal{T}\rightarrow H^1_{\angp}(\op{G}_{S\cup \{p\}}, \g)\rightarrow \bigoplus_{\ell\in S} \frac{H^1(\op{G}_\ell, \g)}{H^1_{\op{nr}}(\op{G}_\ell, \g)},\] from which we obtain the inequality
\[\begin{split}\dim H^1_{\angp}(\op{G}_{S\cup \{p\}}, \g) &\leq \dim \mathcal{T}+\sum_{\ell\in S} \left(\dim H^1(\op{G}_\ell, \g)-\dim H^1_{nr}(\op{G}_\ell, \g)\right)\\
&\leq \dim \mathcal{T}+\sum_{\ell\in S} \left(\dim H^1(\op{G}_\ell, \g)-\dim H^0(\op{G}_\ell, \g)\right)\\
& = \dim \mathcal{T}+\sum_{\ell\in S} \dim H^0(\op{G}_\ell, \g^*),\\
& = \sum_{\ell\in S} \dim H^0(\op{G}_\ell, \g^*),
\end{split}\]
where in the last step, we invoke $\dim \mathcal{T}=0$.
\par Proposition \ref{prop:selmer vs dselmer} asserts that 
\[\dim H^1_{\angp}(\op{G}_{S\cup \{p\}}, \g)=\dim H^1_{\angp}(\op{G}_{S\cup\{p\}}, \g^*)+\sum_{\ell\in S} \dim H^0(\op{G}_\ell, \g^*).\] 
Therefore, $H^1_{\angp}(\op{G}_{S\cup\{p\}}, \g^*)=0$ and 
\[\dim H^1_{\angp}(\op{G}_{S\cup \{p\}}, \g)=\sum_{\ell\in S} \dim H^0(\op{G}_\ell, \g^*).\] 
\end{proof}
\subsection{Congruence primes and vanishing of the Bloch-Kato Selmer group}
\par Next, we discuss conditions for the vanishing of $H^1_\emptyset(\Q, \op{W}_p(E))$. Let $N$ denote the conductor of $E$, and $f$ the Hecke eigencuspform of weight $2$ associated with $E$. We introduce the notion of a congruence prime, which will be key in our analysis of the vanishing of the Selmer group $H^1_{\emptyset}\left(\Q, \op{W}_p(E)\right)$.

\begin{Definition}
We say that $p$ is a \textit{congruence prime} for $f$ if there exists a
newform $f_0$ of weight $2$ and level $d|N$ such that:
\begin{enumerate}
    \item $f_0$ has character lifting the trivial one,
    \item $f_0$ is not Galois conjugate to $f$,
    \item $\bar{\rho}_{f_0,\lambda}\simeq \bar{\rho}_{f, \lambda}$ for some prime $\lambda|p$ of $\bar{\Q}$.
\end{enumerate}
Let $\op{Cong}(f)$ denote the set of congruence primes. A prime $p$ is a \textit{strict} congruence prime if there is a newform satisfying the above conditions with level $N$. Denote by $\op{Cong}_N(f)\subseteq  \op{Cong}(f)$ the subset of strict congruence primes.

\end{Definition}

The following result gives a relationship between congruence primes and the vanishing of the Selmer group $H^1_{\emptyset}\left(\Q, \op{W}_p(E)\right)$, it is largely based on results in \cite{DFG}
\begin{Proposition}\label{prop BK}
Let $p$ be a prime, assume that
that:
\begin{enumerate}
    \item $\bar{\rho}_{f,p}$ is absolutely irreducible,
    \item $p>2$,
    \item either $N>1$ or $p>3$,
    \item $p\nmid N$, 
    \item $\ell \not \equiv 1\mod{p}$ for all primes $\ell|N$,
    \item $\bar{\rho}_{f,p}$ is ramified at all primes $\ell|N$.
\end{enumerate}
Then $H^1_\emptyset(\Q, \op{W}_p(E))\neq 0$ if and only if $p\in \op{Cong}_N(f)$.
\end{Proposition}

\begin{proof}
This is a special case of \cite[Proposition 17]{westonexplicit}. 
\end{proof}
\begin{Lemma}\label{lemma local vanishing}
Suppose that $\ell \in S$, $p\geq 3$ and $\ell\not \equiv \pm 1\mod{p}$. Then, $H^0(\op{G}_\ell, \g^*)=0$ if and only if $\bar{\rho}$ is ramified at $\ell$.
\end{Lemma}
\begin{proof}
We refer the reader to \cite[Lemma 11]{westonexplicit}. 
\end{proof}

\begin{Lemma}\label{weston lemma 12}
    Let $\ell\in S$ be a prime such that $\ell\neq p$ and $\ell\not \equiv \pm 1\mod{p}$ and suppose that $H^0(\op{G}_\ell, \g^*)\neq 0$. Then it follows that $p\in \op{Cong}(f)$.
\end{Lemma}
\begin{proof}
    This result is a consequence of \cite[Lemma 12]{westonexplicit}. 
\end{proof}
\begin{Th}\label{section 4 main thm}
Let $E$ be an elliptic curve defined over $\Q$ with squarefree conductor $N$. Let $\Sigma(E)$ be the set of primes $p$ such that one of the following conditions is satisfied:
\begin{enumerate}
    \item $p\leq 3$,
     \item $\bar{\rho}_{f,p}$ is not absolutely irreducible
    \item $p\mid N $,
    \item there is a prime $\ell|N$, such that $\ell\equiv \pm 1\mod{p}$,
    \item $p\in \op{Cong}(f)$.
\end{enumerate}
Then, for $p\notin \Sigma(E)$, the geometric deformation ring $\mathcal{R}_{E,p}$ is isomorphic to $\Z_p$.
\end{Th}
\begin{proof}
\par
Let $p\notin \Sigma(E)$ and assume by way of contradiction that $\mathcal{R}_{E,p}\not\simeq \Z_p$.  According to Corollary \ref{corollary def zp}, there are two possibilities:
\begin{enumerate}
    \item $H^0(\op{G}_\ell, \g^*)\neq 0$ for some prime $\ell \in S$,
    \item $H^1_{\angp}(\op{G}_{S\cup \{p\}}, \g^*)\neq 0$.
\end{enumerate}
Suppose that the first of the above possibilities does hold. According to Lemma \ref{lemma local vanishing}, the residual representation $\bar{\rho}$ must be unramified at $\ell$. However, according to Lemma \ref{weston lemma 12}, it follows that $p\in \op{Cong}(f)$. However, since $p\notin \Sigma(E)$, it follows therefore that $H^0(\op{G}_\ell, \g^*)=0$ for all primes $\ell \in S$. Therefore, the only possibility is that \[H^1_{\angp}(\op{G}_{S\cup \{p\}}, \g^*)\neq 0.\] Note that since $p\notin \Sigma(E)$, the residual representation $\bar{\rho}$ is absolutely irreducible. Hence, it follows from Proposition \ref{BK to res selmer} that $H^1_\emptyset(\Q, \op{W}_p(E))\neq 0$. Since $H^0(\op{G}_\ell, \g^*)=0$ for all primes $\ell \in S$, it follows from Lemma \ref{lemma local vanishing} that $\bar{\rho}$ is ramified at all primes $\ell\in S$. Proposition \ref{prop BK} then implies that $p\in \op{Cong}_N(f)$, which is not possible since $p\notin \Sigma(E)$. The contradiction implies that $\mathcal{R}_{E,p}\simeq \Z_p$ for all primes $p\notin \Sigma(E)$.
\end{proof}

\begin{Corollary}\label{section 4 last}
Let $E$ be an elliptic curve over $\Q$ with squarefree conductor $N$ and without complex multiplication. Then, for all but finitely many primes $p$, the Galois deformation ring $\mathcal{R}_{E,p}$ is isomorphic to $\Z_p$.
\end{Corollary}
\begin{proof}
It follows from Serre's open image theorem that the set of primes $p$ at which the residual representation $\bar{\rho}_{E,p}$ is not absolutely irreducible is finite. With this it is clear that the set of primes $\Sigma(E)$ in the statement of Theorem \ref{section 4 main thm} is finite and the result follows from the theorem.
\end{proof}
\section{Statistics for a fixed prime $p$ and varying elliptic curve $E_{/\Q}$}\label{section 5}
In this section, we are able to provide some partial answers to the dual problem. Namely, we fix a prime $p$ throughout and let $E$ vary over all non-CM elliptic curves over $\Q$. We will need to assume that $p\geq 5$. 
\subsection{The congruence number and modular degree}
\begin{Definition}\label{section 5 defn}Let $\mathscr{E}_p$ be the set of elliptic curves $E_{/\Q}$ for which the following conditions are satisfied:
\begin{enumerate}
    \item\label{s5:1} $E$ has good reduction at $p$,
    \item\label{s5:2} the conductor $N$ of $E$ is squarefree,
    \item\label{s5:3} the residual representation $\bar{\rho}=\bar{\rho}_{E,p}$ is absolutely irreducible,
    \item\label{s5:4} $\ell\not\equiv \pm 1\mod{p}$ for all primes $\ell\mid N$.
\end{enumerate}
\end{Definition}
\begin{Remark}
It is shown by J. Cremona and M. Sadek in \cite{cremonasadek} that the proportion of elliptic curves over $\Q$ ordered by height satisfying \eqref{s5:1} (resp. \eqref{s5:2}) is $(1-\frac{1}{p})$ (resp. $\zeta(10)/\zeta(2)\approx$ 60.85\%). On the other hand, it is a result of W.Duke that $\bar{\rho}_{E,p}$ is irreducible for $100\%$ of elliptic curves, see \cite{Duke}.
\end{Remark}

\begin{Remark}
We observed computationally for elliptic curves with conductor at most 4000 that the conditions (1), (2) and (4) appear to imply (3).  We are not immediately aware of any obvious reason for this.
\end{Remark}

As $E$ varies in the set $\mathscr{E}_p$, one would like to understand how often is the geometric deformation ring $\mathcal{R}_{E,p}$ isomorphic to $\Z_p$? Let $f$ be the modular form corresponding to $E$. Note that according to Theorem \ref{section 4 main thm}, this is the case when $p\notin \op{Cong}(f)$. Let $\mathscr{E}_p'$ be the subset of $\mathscr{E}_p$ for which this additional condition is satisfied.
\par Let $E$ be an elliptic curve defined over $\Q$, there is a unique minimal Weierstrass equation $y^2=x^3+ax+b$, where $(a,b)\in \Z^2$ is such that $\op{gcd}(a^3 , b^2)$ is not divisible by any twelfth power. The \textit{height of }$E$ is given by $H(E) := \max\left(\abs{a}^3, b^2\right)$.  Ordering elliptic curves over $\Q$  according to height, for any subset $\mathscr{S}\subseteq \mathscr{E}_p$ and $x>0$, let $\mathscr{S}(x)$ consist of all elliptic curves $E\in \mathscr{S}$ with height $\leq x$. The lower density of $\mathscr{S}$ is given by
\[\mathfrak{d}(\mathscr{S}):=\liminf_{x\rightarrow \infty} \frac{\# \mathscr{S}(x)}{\# \mathscr{E}_p(x)}.\]

From a statistical point of view, one would like to characterize $\mathfrak{d}(\mathscr{E}_p')$, i.e., the lower density for the proportion of curves $E\in \mathscr{E}_p$ such that $\mathcal{R}_{E,p}\simeq \Z_p$.
\par We are not able to prove unconditional results, instead, we resort to heuristics and explicit calculations. We recall the notions of \textit{modular degree} and \textit{congruence number} associated to an elliptic curve $E_{/\Q}$.
\begin{Definition}
Let $E_{/\Q}$ be an elliptic curve of conductor $N$ and $\phi_E:X_0(N)\rightarrow E$ the modular homomorphism associated to $E$. The modular degree $m_E$ is the degree of $\phi_E$.
\end{Definition}
Let $f$ be the normalized newform on $\Gamma_0(N)$ associated to $E$. The primes $\op{Cong}(f)$ are exactly those that divide the \emph{congruence number} (cf. \cite{agashe ribet} for the definition). Let $r_E$ the congruence number. 
\par Given $r\in \Z_{\geq 1}$, set $r^{(p)}$ to denote $|r|_p^{-1}$, i.e., the $p$-part of $r$. Since $p\nmid N$, it follows from \cite[Theorem 2.1]{agashe ribet} that $r_E^{(p)}=m_E^{(p)}$. Note that an elliptic curve $E\in \mathscr{E}_p$ belongs to $\mathscr{E}_p'$, precisely when $r_E^{(p)}=1$. Therefore, we may formulate the problem in terms of the distribution of the $p$-primary part of the modular degree
\[\liminf_{x\rightarrow \infty} \frac{\#\mathscr{E}_p'(x)}{\#\mathscr{E}_p(x)}=\liminf_{x\rightarrow \infty} \frac{\#\{E\in \mathscr{E}_p(x)\mid m_E^{(p)}=1\}}{\#\mathscr{E}_p(x)}.\]

The probablity that $p\mid m_E^{(p)}$ for odd primes $p$ is given by Cohen-Lenstra heuristics, see \cite[p.499]{watkins}. It follows from such heuristics that the proportion of curves $m_E^{(p)}=1$ is
\[\prod_i \left(1-\frac{1}{p^{i}}\right)=1-\frac{1}{p}-\frac{1}{p^2}+\frac{1}{p^3}+\dots.\] Watkins obtains some evidence for these heuristics in loc. cit. One can only expect that the stipulation that $m_E^{(p)}=1$ be largely independent from the other conditions \eqref{s5:1}-\eqref{s5:4}.
\par In the rest of this paper, we show through explicit calculation that one may expect that $m_E^{(p)}=1$ holds for a large proportion of curves, especially as $p\rightarrow\infty$. This leads us to expect the following,
\begin{enumerate}
    \item for any odd prime $p$, 
    \[\liminf_{x\rightarrow \infty} \frac{\# \mathscr{E}_p'(x)}{\# \mathscr{E}_p(x)}>0.\]
    \item As $p\rightarrow \infty$, 
     \[\lim_{p\rightarrow \infty}\left(\liminf_{x\rightarrow \infty} \frac{\# \mathscr{E}_p'(x)}{\# \mathscr{E}_p(x)}\right)\rightarrow 1.\] 
\end{enumerate}
\subsection{Statistics for congruence primes}
\par For newforms $f,g$ of levels $M,N$, we say that a rational prime $p$ is a {\it congruence prime} for $f$ and $g$ if there is a prime $\mathfrak p$ of $\bar{\Q}$ containing $p$ such that there is a congruence of Fourier coefficients $$a_\ell(f) \equiv a_{\ell}(g) \pmod{\mathfrak p}$$ for all primes $\ell$ not dividing $pMN$.  

\par Let $E$ be an elliptic curve with conductor $N$ and associated newform $f \in S_2(\Gamma_0(N))$. Recall that a prime $p$ is a {\it congruence prime} for $E$ if there is a weight $2$ newform $g$ of level dividing $N$ such that $p$ is a congruence prime for $f_E$ and $g$. We
say that a congruence prime is {\it strict} (resp.\ {\it proper}) if the level of $g$ can be taken equal to $N$ (resp.\ strictly less than $N$).  In deformation theory, strict congruence primes correspond to global obstructions while proper congruence primes correspond to local
obstructions.

\par The behavior of $2$ as a congruence prime was studied in \cite{CalegariEmerton}, where they proved that it is a congruence prime for the vast majority of eigenforms.  The behavior for odd primes is less clear.  We computed all congruence primes for the 13,352 isogeny
classes of elliptic curves of
conductor at most $4000$.  Computations were done in Magma using the package of Montes to deal with the large number fields which arise as coefficient fields.  (Attempts to extend the
computations to larger conductor were unsuccessful due to the extremely large coefficient
fields which begin to arise at that stage.)

\begin{tabular}{c|ccc}
$p$ & \% proper & \% strict & \% cong \\ \hline
2 &	82.9 &	99.1 &	99.5 \\
3 &	49.1 &	61.7 &	75.4 \\
5 &	16.3 &	23.8 &	37.3 \\
7 &	7.8 &	13.2 &	20.0 \\
11 &	2.5 &	4.8	 &7.2 \\
13 &	1.6 &	3.4	& 4.9 \\
17 &	0.7 &	2.1	& 2.8 \\
19 &	0.5 &	1.3	& 1.8 \\
23 &	0.3 &	0.9	& 1.1 \\
29 &	0 &	0.6	& 0.6 \\
31 &	0.1 &	0.4	& 0.4 \\
37 &	0 &	0.3	& 0.4 \\
41 &	0 &	0.3	& 0.3 \\
43 & 0 &	0.1 &	0.1 \\
47 &	0 &	0.1 &	0.1
\end{tabular}

The proportions for all primes between 53 and 97 were never more than $0.1\%$.

It is difficult to see any immediate pattern in this data beyond the not surprising observation that small congruence primes are relatively common and large congruence primes are rare.  The data for the 179 isogency classes of elliptic curves of prime conductor in our sample (in which case all congruences are necessarily strict) was perhaps more revealing.\

\begin{tabular}{c|ccc}
$p$ & \% cong & \% Watkins $S_3$ & Watkins CL \\ \hline
3 &	40.2 & 44.8 & 44.0\\
5 &	22.9 & 24.2 & 24.0\\
7 &	11.2 & 16.3 & 16.3\\
11 & 9.5 & 9.9 & 9.9\\
13 & 7.3 & 8.1 & 8.3 \\
17 &	3.9	 & 6.2 & 6.2\\
19 &	5.6	 & 5.3 & 5.5\\
23 &	3.4	& 4.7 &  4.5\\
29 &	3.4	& 3.4 & 3.6\\
31 &	2.8	& 3.4 & 3.3\\
37 &	1.1 & 2.7 & 2.8	
\end{tabular}

Here we list also the proportion of elliptic curves $E$ in Watkin's set $S_3$ (consisting of 52878 non-Setzer--Neumann elliptic curves of prime
discriminant of absolute value less than $10^7$, as computed by
\cite{BrumerMcGuinness}) with $p$ dividing the modular degree, as well
as the Cohen--Lenstra prediction Watkins developed for that case.  Given that our data set is quite small, the fit is certainly respectably close.

Since our results are concerned with elliptic curves of squarefree conductor, we also report the proportions in that case, which are notably higher than the overall average, at least for larger $p$.
The data here is for the 4931 isogeny classes of elliptic curves of squarefree conductor $\leq 4000$.

\begin{tabular}{c|c}
$p$ & \% cong \\ \hline
2 &	98.8 	\\
3 &	63.7  \\
5 &	37.9  \\
7 &	21.6  \\
11 & 9.3  \\
13 & 6.8	 \\
17 & 4.1  \\
19 & 2.8	  \\
23 & 1.9	  \\
29 & 1.2	 
\end{tabular}

Finally, and perhaps most relevantly, we note that individual elliptic curves tend to have very few congruence primes.

\begin{tabular}{c|ccc}
$n$ & \% $E$ with $n$ proper cp & \% $E$ with $n$ strict cp & \% $E$ with $n$ cp \\ \hline
0 & 4.3\% & 0.3\% & 0.1\% \\
1 & 40.8\% & 16.6\% & 8.2\% \\
2 & 44.1\% & 55.7\% & 42.2\% \\
3 & 10.2\% & 25.0\% & 38.8\% \\
4 & 0.6\% & 2.3\% & 9.7\% \\
5 & 0.0\% & 0.1\% & 0.9\% \\
6 & 0 & 0 & 0.0\%
\end{tabular}

There were a total of six elliptic curves in the sample with six
congruence primes.

Putting everything together, we computed for small primes $p$ the
proportion of isogeny classes of conductor $\leq 4000$ lying in our 
set $\mathcal{E}_p$, and then also what proportion of those further satisfy
$p \nmid m_E$.

\begin{tabular}{c|cc}
$p$ & \% in ${\mathcal E}_p$ & \% in ${\mathcal E}_p$ with $p \nmid m_E$ \\ \hline
5 & 10.3\% & 6.6\% \\
7 & 17.2\% & 13.5\% \\
11 & 25.1\% & 22.3\% \\
13 & 28.6\% & 26.7\% \\
17 & 30.4\% & 29.2\% \\
19 & 30.9\% & 30.0\% \\
23 & 32.1\% & 31.5\% \\
29 & 33.3\% & 33.0\%
\end{tabular}

\end{document}